\documentclass[12pt]{amsart}
\usepackage[margin=1.2in]{geometry}

\usepackage{amsmath,amssymb,amsfonts,amsthm,bbm,graphicx,color,tikz}

\usepackage{hyperref}
\usepackage{cleveref}
\usepackage{esvect}

\usetikzlibrary{decorations.markings}
\newtheorem{theorem}{Theorem}

\newtheorem{definition}[theorem]{Definition}

\allowdisplaybreaks

\newtheorem{lemma}[theorem]{Lemma}

\newtheorem{problem}[theorem]{Problem}

\DeclareMathOperator{\Real}{Re}
\DeclareMathOperator{\Imaginary}{Im}
\renewcommand{\Re}{\Real}
\renewcommand{\Im}{\Imaginary}
\newcommand{\C}{{\mathbb C}}
\newcommand{\R}{{\mathbb R}}

\newcommand{\G}{{\mathcal G}}
\renewcommand{\H}{{\mathcal H}}
\newcommand{\U}{\mathbf U}

\newcommand{\eps}{\varepsilon}
\newcommand{\be}{\begin{equation}}
\newcommand{\ee}{\end{equation}}
\newcommand{\I}{\mathbf I}

\newcommand{\old}[1]{}

\title[Holomorphic quadratic differentials on graphs]{Holomorphic quadratic differentials on graphs and the chromatic polynomial}
\author{Richard Kenyon, Wai Yeung Lam}

\address{Richard Kenyon\\
	Mathematics Department \\Brown University\\ Providence\\ RI 02912}

\email{richard\_kenyon at brown.edu}

\address{Wai Yeung Lam\\
	Mathematics Department \\Brown University\\ Providence\\ RI 02912}

\email{lam at math.brown.edu}

\begin{document}
\begin{abstract}
We study ``holomorphic quadratic differentials" on graphs. We relate them to the reactive power in an LC circuit, 
and also to the chromatic polynomial of a graph.
Specifically, we show that the chromatic polynomial $\chi$ of a graph $\G$, at negative integer values, can be evaluated as the degree
of a certain rational mapping, arising from the defining equations for a holomorphic quadratic differential.  
This allows us to give an explicit integral expression for $\chi(-k)$.
\end{abstract}

\maketitle

\section{Introduction}
Let $\mathcal{G}=(V,E)$ be a graph together with a specified set $V_{b} \subset V$ of \emph{boundary vertices}. We let $V_{int}= V -V_{b}$ 
be the \emph{interior vertices}.  We assume that every edge connects to at least one interior vertex.

\begin{definition}
A \emph{holomorphic quadratic differential (HQD)} on a graph $\G=(V,E)$ is the data consisting of:
a function $q: E \to \mathbb{R}$ defined on unoriented edges, i.e. $q_{uv} = q_{vu}$, and 
a mapping $z:V \to \mathbb{C}$ to the complex plane, satisfying for every interior vertex $u$
\begin{align}
	\sum_v q_{uv} &=0 \label{qsum}\\
	\sum_v \frac{q_{uv}}{z_u -z_v} &=0.\label{zsum}
\end{align} 
\end{definition}

Holomorphic quadratic differentials on surface graphs arise in discrete differential geometry,
see \cite{Lam2015a} and Section \ref{bkgd} below. 
In some situations we drop condition (\ref{qsum}); see below.
Even though the definition involves both the function $q$ and realization $z$, we refer to $q$ as ``the holomorphic quadratic
differential", since in some applications the realization $z$ is given {\it a priori}. 
In this paper however we are mainly interested in the reversed problem, about existence and enumeration of realizations $z$
for a given $q$.

\begin{problem}\label{2}
Given a function $q:E\to\R$ satisfying (\ref{qsum}), and fixed boundary values 
$z:V_b\to\C$, can we find a realization $z:V\to\C$ satisfying (\ref{zsum})  
and how many such realizations are there?
\end{problem}

In Section \ref{Tuttepolysection} we show, under a genericity assumption on $q$,  that
the number of solutions to Problem \ref{2} satisfies a contraction-deletion relation.
In particular for integer $k\ge 1$,
with certain boundary conditions the number of solutions to Problem \ref{2} is
precisely $|\chi_{\G}(-k)|$ where $\chi_{\G}$ is the chromatic polynomial of $\G$.
This gives a novel way to compute $\chi_{\G}(-k)$ as the number of solutions to a system of rational equations. 
From work in \cite{Abrams2015} it also shows that $\chi_{\G}(-k)$ 
enumerates a certain set
of acyclic orientations of a related graph: see Theorem \ref{orient}. This enumeration appears to be 
distinct from the one given by Stanley in \cite{Stanley2006}, who also enumerated $\chi_{\G}(-k)$ with a set of orientations.

We give an explicit integral formula for $\chi_\G(-k)$ in Theorem \ref{integralthm} below.
Computing $\chi_{\G}(k)$ is generally \#P-complete; for $k\ge 3$, even approximating $\chi_{\G}(k)$ to within a constant factor
(in polynomial time) is NP-hard \cite{Goldberg2008}.
In principle one can approximate $\chi_{\G}$ at negative integer values using the integral formula (\ref{intform}) which has positive integrand. 
However since the integrand is oscillatory
it is not clear how quickly one can approximate this integral for general graphs.

While condition (\ref{qsum}) is important for geometric applications,
since it implies a M\"obius invariance for the set of realizations (see Section \ref{Mobius} below), it is less relevant for the applications
to circuits and colorings which are our concern in this paper. We explain how and why we can drop condition (\ref{qsum}) in certain situations in Section \ref{Mobius} below.

\section{Background}\label{bkgd}

Holomorphic quadratic differentials have a surprising number of geometric applications.
They were introduced in \cite{Lam2015a} for cell decompositions of surfaces. For a triangulated disk in the plane, it was shown there that there is a one-to-one correspondence between holomorphic quadratic differentials, infinitesimal (discrete) conformal deformations, discrete harmonic functions for the cotangent Laplacian and discrete minimal surfaces.  

In the case where the graph is a cell decomposition of a surface, Problem \ref{2} is equivalent to finding a polyhedral surface in space 
(with combinatorics dual to $\G$) with prescribed curvature on edges
\[
q_{uv} = \ell_{uv} \tan \frac{\alpha_{uv}}{2}
\] 
where $\ell$ and $\alpha$ denote edge lengths and dihedral angles, respectively. Indeed, the dihedral angles of such a surface
are determined by the face normals, whose stereographic projection yields a realization $z$ in the plane satisfying (\ref{zsum}). 
The condition $\sum_{v} q_{uv} =0$ indicates that the mean curvature on each face vanishes and thus the polyhedral surface is a 
\emph{discrete minimal surface} \cite{Lam2016}. 

For graphs not associated to surfaces, there are other applications.
For example, without the first condition (\ref{qsum}), the equation (\ref{zsum}) was studied in \cite{Abrams2015}
in the context of the Dirichlet problem on graphs with fixed edge energies.

As another illustration of where HQDs arise in circuits, 
we show below (section \ref{LCcircuitsection}) that holomorphic quadratic differentials (without condition (\ref{qsum}))
arise as the \emph{reactive power} in an LC circuit associated to the graph.

Not every realization $(z_v)_{v\in V}$ of a graph possesses a holomorphic quadratic differential $q$. 
For a fixed realization, finding a holomorphic quadratic differential involves solving a system of linear equations with $E$ unknowns and $3V_{int}$ constraints, and $E-3V_{int}$ can be positive or negative.
Even when it is negative, however, there may be nontrivial solutions. For example, 
realizations of the $\mathbb{Z}^2$-lattice equipped with $q=1$ on horizontal and $q=-1$ on vertical edges include \emph{orthogonal circle patterns}, see \cite{Bobenko1999}.

\section{Preliminaries}

\subsection{M\"obius invariance}\label{Mobius}
Given a realization $(z_v)$ of $\G$ with HQD $(q_{uv})$, and a M\"obius transformation $\phi:\hat\C\to\hat\C$,
The tuple $(\phi(z_v))_{v\in V}$ is also a realization with the same HQD. To see this,
note that it is true when $\phi(z) = az+b$ is affine, so it suffices to prove when $\phi(z)=1/z$ (since every M\"obius
transformation is a composition of affine transformations and inversions).

However when $\phi(z)=1/z$ we can write 
\begin{align*}\sum_v\frac{q_{uv}}{\phi(z_u)-\phi(z_v)}&=-z_u\sum_v\frac{q_{uv}z_v}{z_u-z_v}\\
&=-z_u\sum_v\frac{q_{uv}(z_v-z_u+z_u)}{z_u-z_v}\\
&=-z_u\left(\sum_v q_{uv} + z_u\sum_v\frac{q_{uv}}{z_u-z_v}\right)\\
&=0.\end{align*}

As a special case, suppose $\G$ has a boundary vertex $v_0$ attached to \emph{every} interior vertex, with value $z(v_0)=z_0$ in some realization.
Compose this realization with a M\"obius transformation taking $z_0$ to $\infty$. 
Then we obtain a realization of the smaller graph
$\G\setminus\{v_0\}$ satisfying (\ref{zsum}), but without the condition (\ref{qsum}).
Conversely let $q$ be an \emph{arbitrary} function on the edges of $\G\setminus\{v_0\}$ 
and find a realization of $\G\setminus\{v_0\}$ 
satisfying (\ref{zsum}). Then on $\G$  defining $q_{uv_0}$ for each $u$ so that (\ref{qsum}) is satisfied,
we find a realization of $\G$ with $z_0=\infty$.

\subsection{Space of $q$-values for a given realization}
A graph with boundary $\G$ is said to be \emph{$2$-connected} if every vertex is on a simple 
path between two distinct boundary vertices.

\begin{lemma}\label{fullrank} When $\G$ has a boundary and is $2$-connected, the space of solutions to (\ref{qsum}) is of dimension $|E|-|V_{int}|$.
\end{lemma}

\begin{proof}
The condition (\ref{qsum}) says that $q$ is in the kernel of the incidence matrix $M:\R^E\to\R^{V_{int}}$. 
So it suffices to show that $M$ has full rank $|V_{int}|$, that is, that $M$ is surjective.
Given $\vec{v}\in\R^{V_{int}}$ let us find $q\in\R^E$ with $Mq=\vec{v}$.
It suffices to take $q$ with support on a \emph{wired spanning tree} $T$ of $\G$, that is a spanning forest every component of which 
contains a unique boundary vertex. Now proceed by induction
on the number of internal vertices: if $T$ has an internal leaf $i$, then $q$ on its edge is determined by $v_i$; 
remove this leaf and vertex and finish by induction.
\end{proof}

\subsection{Graph reductions}
We have assumed that $\G$ is a simple graph;  if $\G$ has self-loops then equation (\ref{zsum}) is not defined.

We can allow multiple edges. However if $\G$ has multiple edges then a realization of $\G$ gives a realization of the 
graph $\G'$ 
in which multiple edges have been replaced
by single edges, with the $q$-value on the new edge being the sum of the $q$-values on the multiple edges.

Suppose $\G$ has an interior vertex $v$ of degree $2$, with neighbors $u,w$. Consider a realization $(q,z)$. 
Then at $v$ we have (when $q_{uv}\ne0$)
$$\frac{q_{uv}}{z_v-z_u} - \frac{q_{uv}}{z_v-z_w} = 0,$$
or $z_u=z_w$. In particular a realization of $\G$ is also a realization of $\G'$, the graph in which edge $uv$ and $wv$ have been contracted
(and the values $q_{uv},q_{wv}$ have been discarded).
It is therefore convenient to assume that interior vertices of $\G$ have degree at least $3$.

\subsection{Example with a family of realizations}

When $\G$ is bipartite, and all boundary vertices are on the same side of the bipartition (say, all boundary vertices are white), 
then Problem 2 for certain nongeneric $q$ 
can have a one or more parameter family of solutions. 

\begin{figure}[htbp]
\center{\includegraphics[width=2.5in]{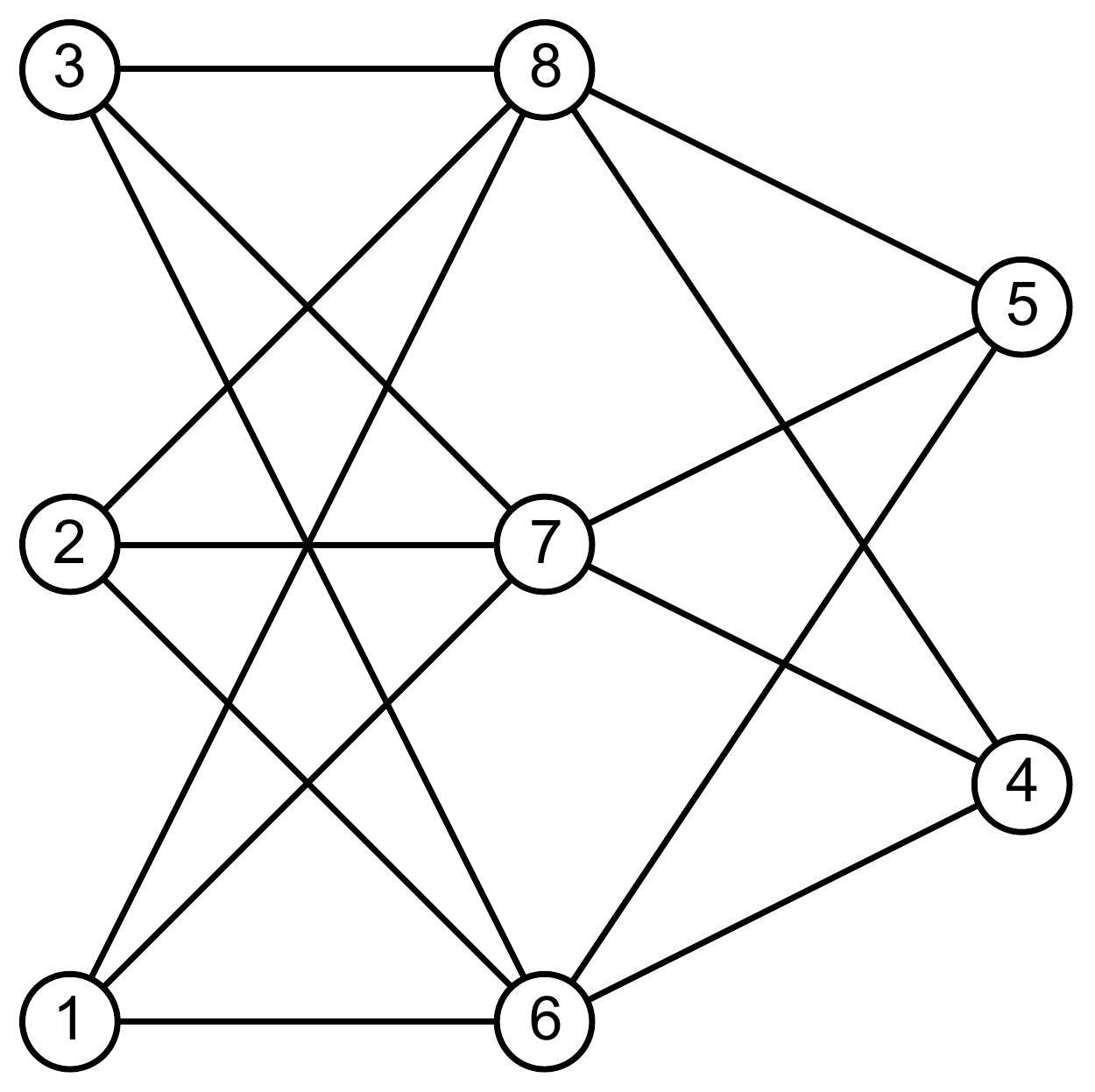}}
\caption{\label{332graph}}
\end{figure}
Consider for example the graph of Figure \ref{332graph}, with boundary vertices $v_1,v_2,v_3$. Suppose that $q$ satisfies
(\ref{qsum}) at all vertices including the boundary vertices. Then we claim that for fixed boundary values $z_1,z_2,z_3$, equation (\ref{zsum})
has a solution in which $z_6=z_7=z_8=b$ for any value of $b\ne z_1,z_2,z_3$. 
The equations at $v_4$ and $v_5$ are trivial, and the equations at $v_6,v_7,v_8$ are
\begin{align}
\label{e1}\frac{q_{64}}{b-z_4}+\frac{q_{65}}{b-z_5}+\frac{q_{61}}{b-z_1}+\frac{q_{62}}{b-z_2}+\frac{q_{63}}{b-z_3}&=0\\
\label{e2}\frac{q_{74}}{b-z_4}+\frac{q_{75}}{b-z_5}+\frac{q_{71}}{b-z_1}+\frac{q_{72}}{b-z_2}+\frac{q_{73}}{b-z_3}&=0\\
\label{e3}\frac{q_{84}}{b-z_4}+\frac{q_{85}}{b-z_5}+\frac{q_{81}}{b-z_1}+\frac{q_{82}}{b-z_2}+\frac{q_{83}}{b-z_3}&=0
\end{align}
and the last is a consequence of the first two (since $q_{6j}+q_{7j}+q_{8j}=0$ for $j=1,2,3,4,5$).
Given $b$ we can solve (\ref{e1}) and (\ref{e2}) for $z_4,z_5$ on condition that $q_{64}q_{75}-q_{65}q_{74}\ne0$,
since they are linear in $\frac1{b-z_4},\frac1{b-z_5}$.

\section{LC-circuits}\label{LCcircuitsection}

We show in this section that a holomorphic quadratic differential satisfying (\ref{zsum}) but not necessarily (\ref{qsum})
can be regarded as the \emph{reactive power} in an AC circuit consisting of inductors and capacitors, where the $z_v$'s represent the complex voltages. 

We consider a circuit with combinatorics $\mathcal{G}=(V,E)$, which is driven at boundary vertices 
by a sinusoidal wave with a common frequency $\omega$ but various phases. 
In order to explain the concept of real power and reactive power, 
we assume for the moment that there are resistors in the circuit as well. 

The voltages $\{U(t)\}$ at the boundary vertices are imposed and satisfy:
\[
U_v(t) = \Re(u_v e^{i\omega t}) 
\]
for some $u:V_b \to \C$ and fixed real $\omega$. 
Let us discuss how to compute the voltages at the interior vertices; these will have the same form for some 
function $u:V\to\C$. 

Under a sinusoidal waveform, associated to a capacitor with capacitance $C>0$, 
an inductor with inductance $L > 0$ and a resistor $R > 0$ is an \emph{impedance}
\begin{align*}
Z_{capacitor} =& \frac{1}{i \omega C} \\
Z_{inductor} =& i \omega L \\
Z_{resistor} =& R.
\end{align*}

We assume each edge consists of some subset of an inductor, a capacitor or a resistor, connected in series. 
The impedance on the edge is the sum of the impedances of the individual elements. We 
thus have $Z:E \to\C$ defined on unoriented edges. 
Note that if no resistor is present, $Z$ is pure imaginary. 

We define at each vertex the \emph{complex voltage} $\U:= u e^{i\omega t}$ and $U = \Re( \U)$. 
Ohm's law relates complex voltage $\U$, complex current $\I$ on edges, and impedance via
\[
\U_v  - \U_w= \I(e_{vw}) Z_{vw}
\]
and Kirchhoff's current law say that at each interior vertex $v$,
$$\sum_{w\sim v} \I(e_{vw})=0,$$
that is, there is no current lost at $v$.
Combining these two laws, we see that $\U$ must satisfy $\Delta\U=0$ at interior vertices for the 
associated Laplacian operator $\Delta:\C^V\to \C^V$ defined by
$$\Delta \U(v) = \sum_{w\sim v}\frac1{Z_{vw}}(\U(v)-\U(w)).$$
Typically $\Delta$ is nonsingular:
a frequency $\omega$ such that $\det\Delta=0$ is a \emph{resonant frequency}. The set of resonant
frequencies is finite  and nonempty if there are no resistances, see Theorem \ref{resonant} below.

For nonresonant frequencies, the function $\U$ exists and is harmonic at interior vertices, that is, $\Delta U(v)=0$ there. 
The \emph{real} current through the oriented edge $e_{uv}$ is
\[
I(e_{vw})= \Re(\mathbf{I}(e_{vw})) = \Re(e^{i \omega t} (u_w - u_v)/ Z_{vw}).  
\]
We have
\begin{align}
U_w - U_v &= \Re(e^{i \omega t} (u_w - u_v)) = |u_w-u_v| \cos(\omega t + \phi) \label{dV}\\
I(e_{vw}) &= \Re(e^{i \omega t} (u_w - u_v)/Z_{vw}) = \frac{|u_w-u_v|}{|Z_{vw}|} \cos(\omega t + \phi+\tilde{\phi})\label{I}
\end{align}
where $Z_{vw} = |Z_{vw}|e^{-i\tilde{\phi}}$. 
We now calculate the \emph{instantaneous power}, by multiplying (\ref{dV}) and (\ref{I}):
\[
P_{vw}(t) := (U_w - U_v) I(e_{vw}) = \frac{|u_w - u_v|^2}{2|Z_{vw}|}\Big( \cos(\tilde{\phi}) (1 + \cos(2\omega t + 2\phi)) - \sin \tilde{\phi} \sin(2 \omega t + 2\phi)\Big).
\]
In terms of \emph{complex power} 
\[
S_{vw}:= (\U_w - \U_v) \overline{\mathbf{I}(e_{vw})}= \frac{|u_w-u_v|^2}{|Z_{vw}|} (\cos \tilde{\phi} - i \sin \tilde{\phi}),
\]
the instantaneous power can be written as
\[
P_{vw}(t) =  \frac{1}{2}( \Re S_{vw} (1 + \cos(2\omega t + 2\phi)) + \Im S_{vw} \, \sin(2\omega t + 2\phi)).
\]
The first part of this equation 
\[
\frac{1}{2} \Re S_{vw} (1 + \cos(2\omega t + 2\phi))
\]
is non-negative and regarded as power dissipated in resistors. In fact if the edge consists of resistors only, it is the same as the instantaneous power since $\Im S =0$. In general, the average power over a period is $\frac12\Re S_{vw}$, which is called \emph{real power}.

On the other hand, the second part of the instantaneous power
\[
\frac{1}{2} \Im S_{vw} \, \sin(2\omega t + 2\phi)
\]
does not contribute any energy over a period. If the edge does not contain any resistors, then this term is exactly the instantaneous power as $\Re S_{vw}=0$. The quantity $\Im S_{vw}$ is called the \emph{reactive power}.

After this explanation of the physical meaning, we can now go back to our original setting where there are no resistors in the circuit and hence no energy loss. In this case the complex power is purely imaginary
\[
S_{vw} = (\U_w - \U_v) \overline{\mathbf{I}(e_{vw})} = |u_w - u_v|^2 / \overline{Z}_{vw} =: iq_{vw}
\]
where $q:E \to \mathbb{R}$ and $q_{vw} = q_{wv}$. Kirchhoff's current law implies that for every interior vertex $v$
\[
0 = \sum_w \mathbf{I}(e_{vw}) = \sum_w -\frac{iq_{vw}}{ \overline{u_w} - \overline{u_v}} e^{i\omega t} 
\]
which is equivalent to
\[
0 = \sum_w \frac{q_{vw}}{u_w - u_v}.
\]

\noindent{\bf Question.} Under what general circumstances will it be the case that (\ref{qsum}) holds (at internal vertices) as well? Is there a ``physical" meaning for such a circuit?

\section{Contraction and deletion}\label{Tuttepolysection}

If $e$ is an edge of $\G$ let $\G-e$ be the graph with edge $e$ deleted and $\G/e$ be the graph with edge $e$ contracted
(identify the two vertices of $e$ and remove $e$). 
We say $q$ satisfying (\ref{qsum}) is \emph{generic} if it is in a Zariski open set, that is, the $q$ are algebraically independent modulo equation
(\ref{qsum}). For example, we can choose $q$ algebraically independently on the complement of a wired spanning tree of $\G$; then the values on the tree
edges are determined uniquely by (\ref{qsum}), as in the proof of Lemma \ref{fullrank}. 

\begin{theorem}\label{contractdelete}
Suppose $\G$ has boundary $v_1,\dots,v_k$ with $k\ge 2$ and every interior vertex is connected to every boundary vertex.
Take a generic function $q:E \to \mathbb{R}$ satisfying (\ref{qsum}).
With distinct boundary values $\{z_k\}_{v_k\in V_b}$ prescribed, and $e$ an edge connecting two interior vertices,
the number $N(\G)$ of distinct solutions $\{z_i\}$ to 
(\ref{zsum}) satisfies the contraction-deletion relation $N(\G) = N(\G-e) + N(\G/e)$.
\end{theorem}

\begin{proof}
We can assume $q_{vw}\in\C$, since the number of solutions $\{z_v\}_{v\in V}$ to 
(\ref{zsum}) is the same whether $q$ is generic over $\C$ or over $\R$. 

We claim that we can move a single interior $z_v$ arbitrarily, 
to get a nearby solution, by changing only the values $q_{vw}$ and $q_{wv_1}$ for neighbors $w$ of $v$. 
Suppose we change $z_v$ to $z_v+\delta$ for small $\delta$, keeping other $z$s fixed; 
we can simultaneously change $q_{vw}$ to $q_{vw}+\eps_w$, $q_{wv_1}$ to $q_{wv_1}-\eps_w$, and 
$q_{vv_1}$ to $q_{vv_1}-\sum_{w\sim v}\eps_w$ for some small $\eps_w$'s so that (\ref{zsum}) is still satisfied. 
To see this, differentiating (\ref{zsum}) at $w$ gives
$$\frac{dq_{vw}}{z_w-z_v} + \frac{q_{vw}dz_v}{(z_w-z_v)^2} + \frac{dq_{wv_1}}{z_w-z_{v_1}},$$
and setting $dq_{wv_1}=-dq_{vw}$ this is zero if
\be\label{dqdz}\frac{dq_{vw}}{dz_v} = \frac{q_{vw}(z_w-z_{v_1})}{(z_v-z_w)(z_v-z_{v_1})}.\ee
Differentiating (\ref{zsum}) at $v$ gives
$$\sum_{w\sim v, w\ne v_1} \left(\frac{dq_{vw}}{z_v-z_w}-\frac{q_{vw}dz_v}{(z_v-z_w)^2}\right)+
\frac{dq_{vv_1}}{z_v-z_{v_1}}-\frac{q_{vv_1}dz_v}{(z_v-z_{v_1})^2}.$$
Substituting $dq_{vv_1} = -\sum dq_{vw}$ this is also zero under condition (\ref{dqdz}). This completes the proof of the claim.

The genericity assumption on $q$s, with the above claim, now implies that for any solution $\{z_v\}_{v\in V}$, 
the values $z_v$ are distinct.

Now let $\{q(t), z_v(t)\}_{t\in[0,1]}$ be a continuous one-parameter family of quadratic holomorphic differentials
and realizations, with $q(0)=q$, with $q_{vw}(t)\ne0$ for $t<1$,
and $q(1)$ equal to zero on edge $e$ and nonzero elsewhere. Boundary $z$-values are fixed.
 
Each solution $\{z_v(t)\}$ varies continuously in $t$, and adjacent $z_v(t)$ are distinct for $t$ sufficiently small.
Under the genericity assumption we can assure that they remain distinct for all $t$: the path
just needs to avoid the (real codimension $2$) subset where some pair of $z$s become equal. 
The only exception is at $t=1$ where it may be the case that $z_v\to z_w$:

Now as $t\to1$, since $q_e=c_e(z_v-z_w)^2\to0$, either the current $c_e(z_v-z_w)\to0$ or the potential drop $z_v-z_w\to0$ (or both). 
We can rule out the case of both happening by genericity (of the approach vector of the path), 
since both going to zero is a complex codimension-two condition.
If $z_v\to z_w$, but $c_e(z_v-z_w)\not\to0$ then we obtain a solution on $\G/e$. If $z_v\not\to z_w$ but $c_e(z_v-z_v)\to0$ 
then we obtain a solution on $\G-e$. 

Conversely let us show that, given a solution on $\G/e$ or $\G-e$, we can construct a nearby solution on $\G$. 

First consider the case of a solution on $\G-e$. By Lemma \ref{harmonic} below and the implicit function theorem, 
starting with a realization of $\G-e$ satisfying $z_v\ne z_w$, we can find a nearby realization of $\G$ for any sufficiently small $q_{vw}$.

Similarly consider a solution on $\G/e$. Lemma \ref{uncontract} below (and the implicit function theorem) 
shows that we can build a nearby solution on $\G$ for small $q_e$. 
\end{proof}

If $z$ is a realization recall the associated Laplacian operator $\Delta$ defined by 
$$\Delta f(v) = \sum_w c_{vw}(f_v-f_w),$$
where $c_{vw} = q_{vw}/(z_v-z_w)^2$. It is invertible for generic $q$.
Note that by (\ref{zsum}), $z$ is harmonic for $\Delta$ at interior vertices.
The inverse of $\Delta$ is the \emph{Green's function} $G=(G_{v,w})_{v,w\in V_{int}}$.

\begin{lemma}\label{harmonic}
For generic $q$, 
let $z$ be a solution to (\ref{zsum}), $v$ a vertex and $e=v_1v_2$ an edge. Then 
\be\label{GG}\frac{dz_v}{dq_e} = \frac{G_{v,v_1}-G_{v,v_2}}{z_{v_1}-z_{v_2}}.\ee
\end{lemma}

The formula still holds when $v=v_1$ or $v=v_2$.

\begin{proof}
Differentiate the equation (\ref{zsum}) with respect to $q_e$ keeping other $q$s fixed. If $v\ne v_1,v_2$ then
$$0= -\sum_{w\sim v}\frac{q_{vw}}{(z_v-z_w)^2}(dz_v-dz_w) = -\sum_{w\sim v}c_{vw}(dz_v-dz_w)$$
so that $dz$ is harmonic at $v$. 
If $v=v_1$ then 
$$0= \frac{dq_{e}}{z_{v_1}-z_{v_2}} - \sum_{w\sim v}c_{vw}(dz_v-dz_w)$$ and similarly at $v=v_2$,
so that $dz$ is harmonic at all vertices except $v_1,v_2$ where it has Laplacian 
$\pm\frac{dq_{e}}{z_{v_1}-z_{v_2}}$ respectively, and the boundary where it is zero. This is precisely (\ref{GG}).
\end{proof}

Note that given a harmonic function on a graph  $\G/e$ with a contracted edge $e$
we can still define the \emph{current across $e$} as the sum of the
currents on edges connecting (in $\G$) to the head of $e$, when these edges are oriented away from $e$.

\begin{lemma}\label{uncontract}
Let $e=v_1v_2$ be an edge of $\G$.
For generic $q$ on $\G/e$, let $z_0$ be a realization of $\G/e$. 
Then for small $q_e$ there is a unique realization $z$ of $\G$ with,  for all $v\in\G$, 
$z_v-z_{0,v}= O(q_e)$. We have $$\frac{dz_v}{dq_e} =\frac{I_e^v}{I_e}$$
where $I_{e}^v$ is the current across $e$ in $\G/e$ from the Green's function $G_{\G/e}(v,\cdot)$
and $I_{e}$ is the current of $z$ across $e$ in $\G/e$ due to the boundary conditions.
\end{lemma}

\begin{proof}
For the graph $\G$ we can rewrite (\ref{GG}) by multiplying the numerator and denominator of the RHS 
by $c_e$. Then the denominator
corresponds to the current across $e$ for the harmonic function $z$. When the edge $e$ is contracted this
becomes the current across $e$ in $\G/e$ which is nonzero
by genericity. (The numerator, $(G_{v,u_2}-G_{v,u_1})c_e$,
is the \emph{transfer current} from $v$ to $e$:
the current across $e$ due to the function with boundary values $z_b$, and which is harmonic on interior vertices of $\G/e$ except at $v$, where one unit of current enters.)
\end{proof}

\begin{theorem}\label{tuttethm}Suppose $\G$ has boundary of size $k\ge 2$ and every interior vertex is connected to every boundary vertex. Let $\G_{int}$ be the induced graph on the interior vertices.
Then the number $N_{\G}$ of realizations $z$ for fixed generic $q$ is $(-1)^{|V_{int}|}\chi(-k+2)$ where $\chi$ 
is the chromatic polynomial of $\G_{int}$.
\end{theorem}

As a special case when the boundary has size $k=2$ there are no realizations.

\begin{proof} As before we send one boundary vertex to $\infty$ using a M\"obius transformation,
and assume that $q$ is positive on the remaining graph, which now has boundary of size $k-1$.

Recall that the Tutte polynomial $T_{\H}(x,y)$ of a graph $\H$ can be defined by the deletion-contraction relation 
$T_{\H}(x,y)=T_{\H-e}(x,y) + T_{\H/e}(x,y)$ 
(for $e$ not a bridge or selfloop)
with the base cases
$T_{\H}(x,y)= x^iy^j$ if $\H$ has $i$ bridges and $j$ selfloops (and no other edges) \cite{Tuttepoly}. 
By Theorem \ref{contractdelete} the quantity $N_{\G}$
satisfies the same recurrence as the Tutte polynomial of $\G_{int}$. However note that
if $\G_{int}$ has a self-loop then $N_{\G}=0$. 
Thus $N_{\G}$ can be identified with a multiple of $T_{\G_{int}}(x,0)$. 

A connected graph with only bridges for edges is a tree,
so it remains to compute $N_{\G}$ when $\G_{int}$ is a tree.
If $\G_{int}$ is a tree, let us see how $N_{\G}$ changes when we delete or contract an edge connected to a leaf.
Contracting an edge connected to a leaf yields a tree with one fewer edge; deleting the edge
yields a single vertex (for which $N_{\{v\}}=k-2$) and a tree with one fewer edge. Thus
\begin{align*}N_{\G} &= N_{\G/e} + N_{\G-e}\\
&= N_{\G/e} + N_{\G/e}N_{\{v\}}\\
&= N_{\G/e} + N_{\G/e}(k-2)
\end{align*} 
from which we arrive at 
$N_{\G} = (k-2)(k-1)^{|E|}$. 
We find $N_{\G} =(k-2)T_{\G_{int}}(k-1,0)$. Finally, for a connected graph,
$\chi_{\G}(-x) = (-1)^{|\G|}xT(1+x,0)$.
\end{proof}

\subsection{Chromatic polynomial and compatible orientations}

For a graph with boundary $V_b$ and a function $u:V_b\to\R$, a \emph{compatible orientation}
is an acyclic orientation, with no internal sources or sinks, with no oriented path from a
boundary vertex to another boundary vertex of higher or equal $u$-value.

In \cite{Abrams2015} it was shown that the number of compatible orientations for a graph with distinct 
boundary values does not depend on the values themselves (or even their relative order).

\begin{theorem}\label{orient} For a graph $\G$ and integer $k\ge 1$, $\chi_{\G}(-k)$ is the number of compatible orientations of $\G_k$,
the graph obtained from $\G$ by adding $k+1$ additional vertices, playing the role of the boundary,
each attached to every vertex of $\G$, and given distinct boundary values.
\end{theorem}

\begin{proof}
From Theorem \ref{tuttethm} above, $|\chi(-k)|$ is the number of solutions to the system (\ref{zsum})
for fixed generic $q$ satisfying (\ref{qsum}). By the remarks in Section \ref{Mobius}, we can 
construct a solution with $1$ fewer boundary vertex by sending its $z$-value to $\infty$.
In this case we can assume without loss of generality that the $q$ on the remaining edges are
positive. Now Theorems 1 and 2 of \cite{Abrams2015} relate the set of solutions with the set of compatible orientations.
\end{proof}

We can use this idea to construct an integral whose value is $|\chi_{\G}(-k)|$, as follows.
Construct the graph $\G_k$ as above, and assign it distinct boundary values $x_0,\dots,x_k\in\R$. 
For any choice of conductances $(c_1,\dots,c_m)$ on the edges of $\G_k$, let $h$ be the harmonic extension,
that is, the solution to the Dirichlet problem with boundary values $x_0,\dots,x_k$.
Let $\Psi:\R^m\to\R^m$ be the map from tuples of edge conductances $(c_1,\dots,c_m)$ of $\G_k$ to tuples of energies 
$(q_1,\dots,q_m)$ for the harmonic extension: here $q_{uv} = c_{uv}(h(u)-h(v))^2$. 
In \cite{Abrams2015} it is shown that the Jacobian of $\Psi$ is
$\pm\prod_e \frac{q_e}{c_e},$ and the degree of $\Psi$ as a map on projective space is the number of compatible orientations.
We can compute the degree of $\Psi$ by integrating its Jacobian. 

Because $\Psi$ is homogeneous and $\Phi^{-1} ((0,\infty) ^E) \subset (0,\infty)^E$ we can simply integrate $\tilde\Psi$
over the simplex $\Delta_m$ of (positive) conductances whose sum is $1$, where $\tilde\Psi=\pi\circ\Psi$ is $\Psi$ projected back to the 
unit simplex. Since the Jacobian of $\pi$ is $Z^{-m}$ where $Z=\sum_e q_e$, the Jacobian of $\tilde\Psi$ is $\prod_{e=uv}\frac{(h(u)-h(v))^2}{Z}$.
We have proved:

\begin{theorem}\label{integralthm} The chromatic polynomial at $-k$ satisfies
\be\label{intform}
|\chi_{\G}(-k)| = \frac1{|\Delta_m|}\int_{\Delta_m} \prod_{e=uv}\frac{(h(u)-h(v))^2}{Z} \,d\mathrm{vol}\ee
where the integral is over the $m$-simplex of conductances on $\G_k$ normalized to sum to $1$, $h$ is the harmonic extension with the given conductances and boundary values $x_0,\dots,x_k$,
$Z=\sum_e q_e$ is the total Dirichlet energy of $h$ and $d\mathrm{vol}$ is the Lebesgue measure on $\Delta_m$.
\end{theorem}

As an example, let $\G$ consist of a single vertex $v$. Then $\G_3$ is a star with three boundary vertices $v_0,v_1,v_2$. 
Assign them boundary values $2,1,0$ respectively
and let their corresponding edge conductances be $c_0,c_1,c_2$ which sum to $1$. The harmonic extension $h$ then 
has $h(v) = 2c_0+c_1$.
The above integral is
\begin{align*}
|\chi_{\G}(-2)|&=2\int_0^1\int_0^{1-c_0}\frac{h(v)^2(h(v)-1)^2(h(v)-2)^2}{(c_2h(v)^2+c_1(h(v)-1)^2+c_0(h(v)-2)^2)^3}dc_1\,dc_0\\
&=2\int_0^1\int_0^{1-c_0}\frac{(2 c_0+c_1-2)^2(2 c_0+c_1-1)^2(2c_0+c_1)^2}{(-4 c_0^2-4 c_1 c_0+4 c_0-c_1^2+c_1)^3}
dc_1\,dc_0\\
   &= 2.\end{align*}
The value of the integral does not depend on the precise boundary values $0,1,2$ we used; can we further simplify
this integral by taking particular boundary values, for example when the boundary values are highly skewed?

\section{Appendix}

\begin{theorem}\label{resonant} In an LC circuit in which not all spanning trees have the same number of inductors, the set of resonant frequencies is finite and nonempty.
\end{theorem}

\begin{proof}
This argument is due to Robin Pemantle.
Recall that a multivariate real polynomial $p(z_1,\dots,z_n)$ is \emph{stable} if $p(z_1,\dots,z_n)\ne 0$ whenever $\Im(z_i)>0$ for each $i$.
See \cite{Wagner2011} for background on stable polynomials.
Stability is preserved under certain operations: setting some variables to be real constants; replacing
a variable $z_i$ by $c z_i$ for $c>0$; setting some variables equal to each other. 

On a finite graph $\G$ with variables $z_e$ on edges, let $Z=Z(\{z_e\})$ be the weighted sum of spanning trees;
it is a homogeneous stable polynomial in the $z_e$, see \cite[Prop. 2.4]{Borcea2008}...in fact this is one of the canonical examples of multivariate stable polynomials.
For those edges $e$ with inductors, replace each $z_e$ by $L_e z_e$ and for
those edges $e$ with capacitors, replace $z_e$ with $1/C_e$. These operations preserve stability 
(since $C_e,L_e>0$). Now set all remaining
factors of $z_e$ to a single value $z$.
This also preserves stability. The new one-variable polynomial of $z$, $\tilde Z(z),$ has only real roots, by stability 
(and the fact that it has real coefficients). Moreover these roots must be negative
since all coefficients are positive. By our hypothesis that not all spanning trees have the same number of inductors, $\tilde Z$ is not a monomial in $z$, so has at least one negative root.
Now replace $z$ with $-\omega^2$, and divide by
$(i\omega)^{n}$ where $n$ is the  degree. This modified (Laurent) polynomial $\tilde Z(-\omega^2)/(i\omega)^n$ 
is the Laplacian determinant. 
\end{proof}

\noindent{\bf Ackowledgements.} We thanks Robin Pemantle for providing the argument of Theorem \ref{resonant}. R. Kenyon
is supported by the NSF grant DMS-1713033 and the Simons Foundation award 327929.

\bibliographystyle{siam}
\bibliography{fixedq}

\end{document}